\documentclass{amsart}
\usepackage{amsfonts}
\usepackage{amsmath}
\usepackage{amssymb}

\newtheorem{theorem}{Theorem}[section]
\newtheorem{lemma}[theorem]{Lemma}
\newtheorem{corollary}[theorem]{Corollary}

\newtheorem{definition}[theorem]{Definition}
\theoremstyle{definition}
\newtheorem{remark}[theorem]{Remark}

\numberwithin{equation}{section}
\input{tcilatex}

\begin{document}
\title[Ostrowski type inequalities]{Ostrowski type inequalities for
functions whose derivatives are preinvex}
\author{\.{I}mdat \.{I}\c{s}can}
\address{\textbf{\.{I}mdat \.{I}\c{s}can}\\
Department of Mathematics, Faculty of Science and Arts, Giresun University,
Giresun, Turkey}
\email{imdat.iscan@giresun.edu.tr}

\begin{abstract}
In this paper, making use of new an identity, we established new
inequalities of Ostrowski's type for the class of preinvex functions and
gave some midpoint type inequalities.
\end{abstract}

\maketitle

\section{Introduction and Preliminaries}

Let $f:I\subset 
\mathbb{R}
\mathbb{\rightarrow R}$ be a differentiable mapping on $I^{\circ }$, the
interior of $I$, and let $\ a,b\in I^{\circ }$ with $a<b.$ If $\left\vert
f^{\prime }(x)\right\vert \leq M$ for all $x\in \left[ a,b\right] $, then
the following inequalities holds:%
\begin{equation}
\left\vert f(x)-\frac{1}{b-a}\dint\limits_{a}^{b}f(u)du\right\vert \leq 
\frac{M}{b-a}\left[ \frac{\left( x-a\right) ^{2}+\left( b-x\right) ^{2}}{2}%
\right] .  \label{1-a}
\end{equation}%
This result well known in the literature as the Ostrowski's inequality \cite[%
p. 469]{MPF94}. For recent results and generalizations concerning
Ostrowski's inequality see \cite{ADDC10}, \cite{KOA11} and the references
therein.

\begin{definition}
The function $f:\left[ a,b\right] \subset \mathbb{R\rightarrow R}$ is said
to be convex if the following inequality holds:%
\begin{equation*}
f\left( tx+(1-t)y\right) \leq tf(x)+(1-t)f(y)
\end{equation*}%
for all $x,y\in \left[ a,b\right] $ and $t\in \left[ 0,1\right] .$ We say
that f is concave if $\left( -f\right) $ is convex.
\end{definition}

The following theorem contains Hadamard's type inequality for M-Lipschitzian
functions. (see \cite{DCK00}).

\begin{theorem}
Let $f:I\subset \mathbb{R\rightarrow R}$ be an M-Lipschitzian mapping on $I,$
and $a,b\in I$ with $a<b.$ Then we have the inequality:%
\begin{equation}
\left\vert f\left( \frac{a+b}{2}\right) -\frac{1}{b-a}\dint%
\limits_{a}^{b}f(x)dx\right\vert \leq M\frac{\left( b-a\right) }{4}.
\label{1-b}
\end{equation}
\end{theorem}

In \cite{K04}, and in \cite{K04b} U.S. Kirmaci proved the following theorems.

\begin{theorem}
Let $f:I^{\circ }\subset \mathbb{R\rightarrow R}$ be a differentiable
mapping on $I^{\circ },$ $a,b\in I^{\circ }$ with $a<b.$ If the mapping $%
\left\vert f^{\prime }\right\vert $ is convex on $\left[ a,b\right] $, then
we have%
\begin{equation}
\left\vert f\left( \frac{a+b}{2}\right) -\frac{1}{b-a}\dint%
\limits_{a}^{b}f(x)dx\right\vert \leq \frac{\left( b-a\right) }{8}\left(
\left\vert f^{\prime }(a)\right\vert +\left\vert f^{\prime }(b)\right\vert
\right)  \label{1-c}
\end{equation}
\end{theorem}

\begin{theorem}
Let $f:I^{\circ }\subset \mathbb{R\rightarrow R}$ be a differentiable
mapping on $I^{\circ },$ $a,b\in I^{\circ }$ with $a<b,$ and let $p>1.$ If
the mapping $\left\vert f^{\prime }\right\vert ^{\frac{p}{p-1}}$ is convex
on $\left[ a,b\right] $, then we have%
\begin{eqnarray}
&&\left\vert f\left( \frac{a+b}{2}\right) -\frac{1}{b-a}\dint%
\limits_{a}^{b}f(x)dx\right\vert \leq \frac{b-a}{16}\left( \frac{4}{p+1}%
\right) ^{\frac{1}{p}}  \label{1-d} \\
&&\times \left\{ \left( 3\left\vert f^{\prime }(a)\right\vert ^{\frac{p}{p-1}%
}+\left\vert f^{\prime }\left( b\right) \right\vert ^{\frac{p}{p-1}}\right)
^{\frac{p-1}{p}}+\left( 3\left\vert f^{\prime }(b)\right\vert ^{\frac{p}{p-1}%
}+\left\vert f^{\prime }\left( a\right) \right\vert ^{\frac{p}{p-1}}\right)
^{\frac{p-1}{p}}\right\} .  \notag
\end{eqnarray}%
$\ \ \ \ \ \ \ \ \ \ \ \ \ \ \ \ $
\end{theorem}

\begin{theorem}
Let $f:I^{\circ }\subset \mathbb{R\rightarrow R}$ be a differentiable
mapping on $I^{\circ },$ $a,b\in I^{\circ }$ with $a<b,$ and let $p>1.$ If
the mapping $\left\vert f^{\prime }\right\vert ^{\frac{p}{p-1}}$ is convex
on $\left[ a,b\right] $, then we have%
\begin{equation}
\left\vert f\left( \frac{a+b}{2}\right) -\frac{1}{b-a}\dint%
\limits_{a}^{b}f(x)dx\right\vert \leq \frac{b-a}{4}\left( \frac{4}{p+1}%
\right) ^{\frac{1}{p}}\left( \left\vert f^{\prime }(a)\right\vert
+\left\vert f^{\prime }(b)\right\vert \right) .  \label{1-e}
\end{equation}
\end{theorem}

\begin{theorem}
Let $f:I^{\circ }\subset \mathbb{R\rightarrow R}$ be a differentiable
mapping on $I^{\circ },$ $a,b\in I^{\circ }$ with $a<b,$ and let $p>1.$ If
the mapping $\left\vert f^{\prime }\right\vert ^{p}$ is convex on $\left[ a,b%
\right] $, then 
\begin{equation}
\left\vert f\left( \frac{a+b}{2}\right) -\frac{1}{b-a}\dint%
\limits_{a}^{b}f(x)dx\right\vert \leq \left( \frac{3^{1-\frac{1}{p}}}{8}%
\right) \left( b-a\right) \left( \left\vert f^{\prime }(a)\right\vert
+\left\vert f^{\prime }(b)\right\vert \right)  \label{1-ee}
\end{equation}
\end{theorem}

In recent years several extentions and generalizations have been considered
for classical convexity. A significant generalization of convex functions is
that of invex functions introduced by Hanson in \cite{H81}. Weir and Mond 
\cite{WM98} introduced the concept of preinvex functions and applied it to
the establisment of the sufficient optimality conditions and duality in
nonlinear programming. Pini \cite{P91} introduced the concept of
prequasiinvex function as a generalization of invex functions. Later, Mohan
and Neogy \cite{MN95} obtained some properties of generalized preinvex
functions. Noor \cite{N07b} has established some Hermite-Hadamard type
inequalities for preinvex and log-preinvex functions.

The aim of this paper is to establish some Ostrowski type inequalities for
functions whose derivatives in absolute value are preinvex. Now we recall
some notions in invexity analysis which will be used throught the paper (see 
\cite{A05,MG08,YL01} and references therein)

Let $f:A\mathbb{\rightarrow R}$ and $\eta :A\times A\rightarrow 
\mathbb{R}
,$where $A$ is a nonempty set in $%
\mathbb{R}
^{n}$, be continuous functions.

\begin{definition}
The set $A\subset 
\mathbb{R}
^{n}$ is said to be invex with respect to $\eta (.,.)$, if for every $x,y\in
A$ and $t\in \left[ 0,1\right] ,$%
\begin{equation*}
x+t\eta (y,x)\in A.
\end{equation*}

The invex set $A$ is also called a $\eta -$connected set.
\end{definition}

It is obvious that every convex set is invex with respect to $\eta (y,x)=y-x$%
, but there exist invex sets which are not convex \cite{A05}.

\begin{definition}
The function $f$ on the invex set $A$ is said to be preinvex with respect to 
$\eta $ if 
\begin{equation*}
f\left( x+t\eta (y,x)\right) \leq \left( 1-t\right) f(x)+tf(y),\ \forall
x,y\in A,\ t\in \left[ 0,1\right] .
\end{equation*}

The function $f$ is said to be preconcave if and only if $-f$ \ is preinvex.
\end{definition}

\bigskip Note that every convex function is a preinvex function, but the
converse is not true \cite{MG08}. For example $f(x)=-\left\vert x\right\vert
,\ x\in 
\mathbb{R}
,$ is not a convex function, but it is a preinvex function with respect to 
\begin{equation*}
\eta (x,y)=\left\{ 
\begin{array}{cc}
x-y, & if\ \ xy\geq 0 \\ 
y-x, & if\ \ xy<0%
\end{array}%
\right. .
\end{equation*}

We also need the following assumption regarding the function $\eta $ which
is due to Mohan and Neogy \cite{MN95}:

\textbf{Condition C:} Let $A\subset $ $%
\mathbb{R}
^{n}$ be an open invex subset with respect to $\eta :A\times A\rightarrow 
\mathbb{R}
.$ For any $x,y\in A$ and any $t\in \left[ 0,1\right] ,$%
\begin{eqnarray*}
\eta \left( y,y+t\eta (x,y)\right) &=&-t\eta (x,y) \\
\eta \left( x,y+t\eta (x,y)\right) &=&(1-t)\eta (x,y).
\end{eqnarray*}

Note that for every $x,y\in A$ and every $t_{1},t_{2}\in \left[ 0,1\right] $
from condition C, we have%
\begin{equation*}
\eta \left( y+t_{2}\eta (x,y),y+t_{1}\eta (x,y)\right) =(t_{2}-t_{1})\eta
(x,y).
\end{equation*}

There are many vector functions that satisfy condition C \cite{MG08},
besides the trivial case $\eta (x,y)=x-y.$ For example let $A=%
\mathbb{R}
\backslash \left\{ 0\right\} $ and%
\begin{equation*}
\eta (x,y)=\left\{ 
\begin{array}{cc}
x-y, & if\ \ x>0,\ y>0 \\ 
x-y, & if\ \ x<0,\ y<0 \\ 
-y, & otherwise.%
\end{array}%
\right.
\end{equation*}%
Then $A$ is an invex set and $\eta $ satisfies condition C.

\section{Main Results}

\begin{lemma}
\label{2.1}Let $A\subset $ $%
\mathbb{R}
$ be an open invex subset with respect to $\eta :A\times A\rightarrow 
\mathbb{R}
$ and $a,b\in A$ with $a<a+\eta (b,a).$ Suppose that $f:A\rightarrow 
\mathbb{R}
$ is a differentiable function. If $f^{\prime }$ is integrable on $\left[
a,a+\eta (b,a)\right] $, then the following equality holds:%
\begin{eqnarray*}
&&f(x)-\frac{1}{\eta (b,a)}\dint\limits_{a}^{a+\eta (b,a)}f(u)du \\
&=&\eta (b,a)\left( \dint\limits_{0}^{\frac{x-a}{\eta (b,a)}}tf^{\prime
}\left( a+t\eta (b,a)\right) dt+\dint\limits_{\frac{x-a}{\eta (b,a)}%
}^{1}\left( t-1\right) f^{\prime }\left( a+t\eta (b,a)\right) dt\right)
\end{eqnarray*}%
for all $x\in \left[ a,a+\eta (b,a)\right] $.
\end{lemma}

Since $A\subset $ $%
\mathbb{R}
$ is an invex subset with respect to $\eta :A\times A\rightarrow 
\mathbb{R}
$ and $a,b\in A$, for all $t\in \left[ 0,1\right] $ we have $a+t\eta
(b,a)\in A.$ A simple proof of equality can be given by performing an
integration by parts in the integrals from the right side and changing the
variable. The details are left to the interested reader.

The following result may be stated:

\begin{theorem}
\label{2.2}Let $A\subset $ $%
\mathbb{R}
$ be an open invex subset with respect to $\eta :A\times A\rightarrow 
\mathbb{R}
$ and $a,b\in A$ with $a<a+\eta (b,a).$ Suppose that $f:A\rightarrow 
\mathbb{R}
$ is a differentiable function and $\left\vert f^{\prime }\right\vert $ is
preinvex function on $A$. If $f^{\prime }$ is integrable on $\left[ a,a+\eta
(b,a)\right] $,then the following inequality holds:%
\begin{eqnarray}
&&\left\vert f(x)-\frac{1}{\eta (b,a)}\dint\limits_{a}^{a+\eta
(b,a)}f(u)du\right\vert \leq \frac{\eta (b,a)}{6}  \label{2-1} \\
&&\times \left\{ \left[ 3\left( \frac{x-a}{\eta (b,a)}\right) ^{2}-2\left( 
\frac{x-a}{\eta (b,a)}\right) ^{3}+2\left( \frac{a+\eta (b,a)-x}{\eta (b,a)}%
\right) ^{3}\right] \left\vert f^{\prime }(a)\right\vert \right.  \notag \\
&&\left. +\left[ 1-3\left( \frac{x-a}{\eta (b,a)}\right) ^{2}+4\left( \frac{%
x-a}{\eta (b,a)}\right) ^{3}\right] \left\vert f^{\prime }(b)\right\vert
\right\}  \notag
\end{eqnarray}%
for all $x\in \left[ a,a+\eta (b,a)\right] $. The constant $\frac{1}{6}$ is
best possible in the sense that it cannot be replaced by a smaller value.
\end{theorem}

\begin{proof}
By lemma \ref{2.1} and since $\left\vert f^{\prime }\right\vert $ is
preinvex, we have%
\begin{eqnarray*}
&&\left\vert f(x)-\frac{1}{\eta (b,a)}\dint\limits_{a}^{a+\eta
(b,a)}f(u)du\right\vert \\
&\leq &\eta (b,a)\left\{ \dint\limits_{0}^{\frac{x-a}{\eta (b,a)}%
}t\left\vert f^{\prime }\left( a+t\eta (b,a)\right) \right\vert
dt+\dint\limits_{\frac{x-a}{\eta (b,a)}}^{1}\left( 1-t\right) \left\vert
f^{\prime }\left( a+t\eta (b,a)\right) \right\vert dt\right\} \\
&\leq &\eta (b,a)\left\{ \dint\limits_{0}^{\frac{x-a}{\eta (b,a)}}t\left[
\left( 1-t\right) \left\vert f^{\prime }(a)\right\vert +t\left\vert
f^{\prime }(b)\right\vert \right] dt+\dint\limits_{\frac{x-a}{\eta (b,a)}%
}^{1}\left( 1-t\right) \left[ \left( 1-t\right) \left\vert f^{\prime
}(a)\right\vert +t\left\vert f^{\prime }(b)\right\vert \right] dt\right\} \\
&\leq &\eta (b,a)\left\{ \left[ \frac{1}{2}\left( \frac{x-a}{\eta (b,a)}%
\right) ^{2}-\frac{1}{3}\left( \frac{x-a}{\eta (b,a)}\right) ^{3}+\frac{1}{3}%
\left( \frac{a+\eta (b,a)-x}{\eta (b,a)}\right) ^{3}\right] \left\vert
f^{\prime }(a)\right\vert \right. \\
&&+\left. \left[ \frac{1}{6}-\frac{1}{2}\left( \frac{x-a}{\eta (b,a)}\right)
^{2}+\frac{2}{3}\left( \frac{x-a}{\eta (b,a)}\right) ^{3}\right] \left\vert
f^{\prime }(b)\right\vert \right\}
\end{eqnarray*}%
where we have used the fact that%
\begin{eqnarray*}
&&\dint\limits_{0}^{\frac{x-a}{\eta (b,a)}}t\left( 1-t\right)
dt+\dint\limits_{\frac{x-a}{\eta (b,a)}}^{1}\left( 1-t\right) ^{2}dt \\
&=&\frac{1}{2}\left( \frac{x-a}{\eta (b,a)}\right) ^{2}-\frac{1}{3}\left( 
\frac{x-a}{\eta (b,a)}\right) ^{3}+\frac{1}{3}\left( \frac{a+\eta (b,a)-x}{%
\eta (b,a)}\right) ^{3}
\end{eqnarray*}%
and%
\begin{equation*}
\dint\limits_{0}^{\frac{x-a}{\eta (b,a)}}t^{2}dt+\dint\limits_{\frac{x-a}{%
\eta (b,a)}}^{1}t\left( 1-t\right) dt=\frac{1}{6}-\frac{1}{2}\left( \frac{x-a%
}{\eta (b,a)}\right) ^{2}+\frac{2}{3}\left( \frac{x-a}{\eta (b,a)}\right)
^{3}.
\end{equation*}%
To prove that the constant $\frac{1}{6}$ is best possible, let us assume
that (\ref{2-1}) holds with constant $K>0,$ i.e.,%
\begin{eqnarray*}
&&\left\vert f(x)-\frac{1}{\eta (b,a)}\dint\limits_{a}^{a+\eta
(b,a)}f(u)du\right\vert \leq K.\eta (b,a) \\
&&\times \left\{ \left[ 3\left( \frac{x-a}{\eta (b,a)}\right) ^{2}-2\left( 
\frac{x-a}{\eta (b,a)}\right) ^{3}+2\left( \frac{a+\eta (b,a)-x}{\eta (b,a)}%
\right) ^{3}\right] \left\vert f^{\prime }(a)\right\vert \right. \\
&&+\left[ 1-3\left( \frac{x-a}{\eta (b,a)}\right) ^{2}+4\left( \frac{x-a}{%
\eta (b,a)}\right) ^{3}\right] \left\vert f^{\prime }(b)\right\vert .
\end{eqnarray*}%
Let $f(x)=x,$ and then set $x=a+\eta (b,a),$ we get%
\begin{equation*}
\frac{\eta (b,a)}{2}\leq 3K.\eta (b,a),
\end{equation*}%
which gives $K\geq \frac{1}{6},$ and the inequality (\ref{2-1}) is proved.
\end{proof}

\begin{remark}
Suppose that all the assumptions of Theorem \ref{2.2} are satisfied.

(a) If we choose $\eta (b,a)=b-a$ and $x=\frac{2a+\eta (b,a)}{2}$, then we
obtain the inequality%
\begin{equation}
\left\vert f\left( \frac{a+b}{2}\right) -\frac{1}{b-a}\dint%
\limits_{a}^{b}f(u)du\right\vert \leq \frac{b-a}{8}\left( \left\vert
f^{\prime }(a)\right\vert +\left\vert f^{\prime }(b)\right\vert \right)
\label{2-a}
\end{equation}%
which is the same with the inequality (\ref{1-c}).

(b) In (a) with $\left\vert f^{\prime }(x)\right\vert \leq M,~M>0,$ we get
the inequality%
\begin{equation*}
\left\vert f\left( \frac{a+b}{2}\right) -\frac{1}{b-a}\dint%
\limits_{a}^{b}f(u)du\right\vert \leq M\frac{\left( b-a\right) }{4}
\end{equation*}%
which is the same with the inequality (\ref{1-b}).

(c) If the mapping $\eta $ satisfies condition C then by use of the
preinvexity of $\left\vert f^{\prime }\right\vert $ we get 
\begin{eqnarray}
\left\vert f^{\prime }\left( a+t\eta (b,a)\right) \right\vert &=&\left\vert
f^{\prime }\left( a+\eta (b,a)+(1-t)\eta (a,a+\eta (b,a))\right) \right\vert
\notag \\
&\leq &t\left\vert f^{\prime }\left( a+\eta (b,a)\right) \right\vert
+(1-t)\left\vert f^{\prime }(a)\right\vert .  \label{2-11}
\end{eqnarray}%
for every $t\in \left[ 0,1\right] .$

Using the inequality (\ref{2-11}) in the proof of Theorem \ref{2.2}, then
the inequality (\ref{2-1}) becomes the following inequality:%
\begin{eqnarray}
&&\left\vert f(x)-\frac{1}{\eta (b,a)}\dint\limits_{a}^{a+\eta
(b,a)}f(u)du\right\vert \leq \frac{\eta (b,a)}{6}  \label{2-b} \\
&&\times \left\{ \left[ 3\left( \frac{x-a}{\eta (b,a)}\right) ^{2}-2\left( 
\frac{x-a}{\eta (b,a)}\right) ^{3}+2\left( \frac{a+\eta (b,a)-x}{\eta (b,a)}%
\right) ^{3}\right] \left\vert f^{\prime }(a)\right\vert \right.  \notag \\
&&\left. +\left[ 1-3\left( \frac{x-a}{\eta (b,a)}\right) ^{2}+4\left( \frac{%
x-a}{\eta (b,a)}\right) ^{3}\right] \left\vert f^{\prime }(a+\eta
(b,a))\right\vert \right\}  \notag
\end{eqnarray}%
for all $x\in \left[ a,a+\eta (b,a)\right] .$ We note that by use of the
preinvexity of $\left\vert f^{\prime }\right\vert $ we have%
\begin{equation*}
\left\vert f^{\prime }(a+\eta (b,a))\right\vert \leq \left\vert f^{\prime
}(b)\right\vert .
\end{equation*}%
Therefore, the inequality (\ref{2-b}) is better than the inequality (\ref%
{2-1}).
\end{remark}

\begin{theorem}
\label{2.3}Let $A\subset $ $%
\mathbb{R}
$ be an open invex subset with respect to $\eta :A\times A\rightarrow 
\mathbb{R}
$ and $a,b\in A$ with $a<a+\eta (b,a).$ Suppose that $f:A\rightarrow 
\mathbb{R}
$ is a differentiable function such that $\left\vert f^{\prime }\right\vert
^{q}$ is preinvex function on $\left[ a,a+\eta (b,a)\right] $ for some fixed 
$q>1$. If $f^{\prime }$ is integrable on $\left[ a,a+\eta (b,a)\right] $ and 
$\eta $ satisfies condition C, then for each $x\in \left[ a,a+\eta (b,a)%
\right] $ the following inequality holds:%
\begin{eqnarray}
&&\ \ \ \ \ \ \ \ \ \ \ \ \ \ \ \ \ \ \ \ \left\vert f(x)-\frac{1}{\eta (b,a)%
}\dint\limits_{a}^{a+\eta (b,a)}f(u)du\right\vert  \label{2-2} \\
&\leq &\left( \frac{1}{p+1}\right) ^{\frac{1}{p}}\left\{ \frac{\left(
x-a\right) ^{2}}{\eta (b,a)}\left( \frac{\left\vert f^{\prime
}(a)\right\vert ^{q}+\left\vert f^{\prime }\left( x\right) \right\vert ^{q}}{%
2}\right) ^{\frac{1}{q}}\right.  \notag \\
&&\left. +\frac{\left( a+\eta (b,a)-x\right) ^{2}}{\eta (b,a)}\left( \frac{%
\left\vert f^{\prime }(a+\eta (b,a))\right\vert ^{q}+\left\vert f^{\prime
}\left( x\right) \right\vert ^{q}}{2}\right) ^{\frac{1}{q}}\right\}  \notag
\end{eqnarray}%
where $\frac{1}{p}+\frac{1}{q}=1.$
\end{theorem}

\begin{proof}
We first note that if $\left\vert f^{\prime }\right\vert ^{q}$ is a preinvex
function on $\left[ a,a+\eta (b,a)\right] $ and the mapping $\eta $
satisfies condition C then for every $t\in \left[ 0,1\right] ,$it yields the
inequality (\ref{2-11}) and similarly%
\begin{eqnarray}
\left\vert f^{\prime }\left( a+(1-t)\eta (b,a)\right) \right\vert ^{q}
&=&\left\vert f^{\prime }\left( a+\eta (b,a)+t\eta (a,a+\eta (b,a))\right)
\right\vert ^{q}  \notag \\
&\leq &(1-t)\left\vert f^{\prime }\left( a+\eta (b,a)\right) \right\vert
^{q}+t\left\vert f^{\prime }(a)\right\vert ^{q}.  \label{2-3}
\end{eqnarray}

By adding these inequalities we have

\begin{equation}
\left\vert f^{\prime }\left( a+t\eta (b,a)\right) \right\vert
^{q}+\left\vert f^{\prime }\left( a+(1-t)\eta (b,a)\right) \right\vert
^{q}\leq \left\vert f^{\prime }(a)\right\vert ^{q}+\left\vert f^{\prime
}\left( a+\eta (b,a)\right) \right\vert ^{q}  \label{2-4}
\end{equation}

Then integrating the inequality (\ref{2-4}) with respect to $t$ over $\left[
0,1\right] ,$ we obtain%
\begin{equation}
\frac{1}{\eta (b,a)}\dint\limits_{a}^{a+\eta (b,a)}\left\vert f^{\prime
}(t)\right\vert ^{q}dt\leq \frac{\left\vert f^{\prime }(a)\right\vert
^{q}+\left\vert f^{\prime }\left( a+\eta (b,a)\right) \right\vert ^{q}}{2}
\label{2-5}
\end{equation}%
From lemma \ref{2.1} and using H\"{o}lder inequality, we have%
\begin{eqnarray*}
&&\left\vert f(x)-\frac{1}{\eta (b,a)}\dint\limits_{a}^{a+\eta
(b,a)}f(u)du\right\vert \\
&\leq &\eta (b,a)\left( \dint\limits_{0}^{\frac{x-a}{\eta (b,a)}%
}t^{p}dt\right) ^{\frac{1}{p}}\left( \dint\limits_{0}^{\frac{x-a}{\eta (b,a)}%
}\left\vert f^{\prime }\left( a+t\eta (b,a)\right) \right\vert ^{q}dt\right)
^{\frac{1}{q}} \\
&&+\eta (b,a)\left( \dint\limits_{\frac{x-a}{\eta (b,a)}}^{1}\left(
1-t\right) ^{p}dt\right) ^{\frac{1}{p}}\left( \dint\limits_{\frac{x-a}{\eta
(b,a)}}^{1}\left\vert f^{\prime }\left( a+t\eta (b,a)\right) \right\vert
^{q}dt\right) ^{\frac{1}{q}} \\
&\leq &\left( \frac{1}{p+1}\right) ^{\frac{1}{p}}\left\{ \frac{\left(
x-a\right) ^{2}}{\eta (b,a)}\left( \frac{\left\vert f^{\prime }\left(
a\right) \right\vert ^{q}+\left\vert f^{\prime }\left( x\right) \right\vert
^{q}}{2}\right) ^{\frac{1}{q}}\right. \\
&&\left. +\frac{\left( a+\eta (b,a)-x\right) ^{2}}{\eta (b,a)}\left( \frac{%
\left\vert f^{\prime }\left( a+\eta (b,a)\right) \right\vert ^{q}+\left\vert
f^{\prime }\left( x\right) \right\vert ^{q}}{2}\right) ^{\frac{1}{q}}\right\}
\end{eqnarray*}%
where we use the fact that%
\begin{equation*}
\dint\limits_{0}^{\frac{x-a}{\eta (b,a)}}t^{p}dt=\frac{1}{p+1}\left( \frac{%
x-a}{\eta (b,a)}\right) ^{p+1},\ \ \dint\limits_{\frac{x-a}{\eta (b,a)}%
}^{1}\left( 1-t\right) ^{p}dt=\frac{1}{p+1}\left( \frac{a+\eta (b,a)-x}{\eta
(b,a)}\right) ^{p+1},
\end{equation*}%
and by (\ref{2-5}) we get%
\begin{equation*}
\dint\limits_{0}^{\frac{x-a}{\eta (b,a)}}\left\vert f^{\prime }\left(
a+t\eta (b,a)\right) \right\vert ^{q}dt\leq \frac{x-a}{\eta (b,a)}\left( 
\frac{\left\vert f^{\prime }\left( a\right) \right\vert ^{q}+\left\vert
f^{\prime }\left( x\right) \right\vert ^{q}}{2}\right) ,
\end{equation*}%
\begin{equation*}
\dint\limits_{\frac{x-a}{\eta (b,a)}}^{1}\left\vert f^{\prime }\left(
a+t\eta (b,a)\right) \right\vert ^{q}dt\leq \frac{a+\eta (b,a)-x}{\eta (b,a)}%
\left( \frac{\left\vert f^{\prime }\left( a+\eta (b,a)\right) \right\vert
^{q}+\left\vert f^{\prime }\left( x\right) \right\vert ^{q}}{2}\right) .
\end{equation*}
\end{proof}

\begin{corollary}
Suppose that all the assumptions of Theorem \ref{2.3} are satisfied. If we
choose $\left\vert f^{\prime }\left( x\right) \right\vert \leq M,~M>0,$ for
each $x\in \left[ a,a+\eta (b,a)\right] ,$ then we have%
\begin{equation*}
\left\vert f(x)-\frac{1}{\eta (b,a)}\dint\limits_{a}^{a+\eta
(b,a)}f(u)du\right\vert \leq \left( \frac{1}{p+1}\right) ^{\frac{1}{p}%
}M\left( \frac{\left( x-a\right) ^{2}+\left( a+\eta (b,a)-x\right) ^{2}}{%
\eta (b,a)}\right) .
\end{equation*}
\end{corollary}

\begin{corollary}
\label{s1} Suppose that all the assumptions of Theorem \ref{2.3} are
satisfied. If we choose $x=\frac{2a+\eta (b,a)}{2}$, then using the
preinvexity of $\left\vert f^{\prime }\right\vert ^{q}$ on $\left[ a,a+\eta
(b,a)\right] $ and\ by inequality (\ref{2-11}) for $t=\frac{1}{2}$ we have%
\begin{eqnarray*}
&&\left\vert f\left( \frac{2a+\eta (b,a)}{2}\right) -\frac{1}{\eta (b,a)}%
\dint\limits_{a}^{a+\eta (b,a)}f(u)du\right\vert \\
&\leq &\left( \frac{1}{p+1}\right) ^{\frac{1}{p}}\left\{ \frac{\eta (b,a)}{4}%
\left( \frac{3\left\vert f^{\prime }(a)\right\vert ^{q}+\left\vert f^{\prime
}\left( a+\eta (b,a)\right) \right\vert ^{q}}{4}\right) ^{\frac{1}{q}}\right.
\\
&&\left. +\frac{\eta (b,a)}{4}\left( \frac{3\left\vert f^{\prime }(a+\eta
(b,a))\right\vert ^{q}+\left\vert f^{\prime }\left( a\right) \right\vert ^{q}%
}{4}\right) ^{\frac{1}{q}}\right\}
\end{eqnarray*}
\end{corollary}

\begin{remark}
In Corollary \ref{s1}, if we take $\eta (b,a)=b-a,$ then we have the
inequality%
\begin{eqnarray*}
&&\left\vert f\left( \frac{a+b}{2}\right) -\frac{1}{b-a}\dint%
\limits_{a}^{b}f(u)du\right\vert \\
&\leq &\frac{b-a}{16}\left( \frac{4}{p+1}\right) ^{\frac{1}{p}}\left\{
\left( 3\left\vert f^{\prime }(a)\right\vert ^{q}+\left\vert f^{\prime
}\left( b\right) \right\vert ^{q}\right) ^{\frac{1}{q}}+\left( 3\left\vert
f^{\prime }(b)\right\vert ^{q}+\left\vert f^{\prime }\left( a\right)
\right\vert ^{q}\right) ^{\frac{1}{q}}\right\}
\end{eqnarray*}%
which is the same with the inequality (\ref{1-d}). Let $a_{1}=3\left\vert
f^{\prime }(a)\right\vert ^{q},\ b_{1}=\left\vert f^{\prime }\left( b\right)
\right\vert ^{q},\ a_{2}=3\left\vert f^{\prime }(b)\right\vert ^{q},\
b_{2}=\left\vert f^{\prime }\left( a\right) \right\vert ^{q}.$ Here $0<\frac{%
1}{q}<1,$ for $q>1.$Using the fact that%
\begin{equation}
\dsum\limits_{k=1}^{n}\left( a_{k}+b_{k}\right) ^{s}\leq
\dsum\limits_{k=1}^{n}a_{k}^{s}+\dsum\limits_{k=1}^{n}b_{k}^{s}  \label{2-6}
\end{equation}%
for $\left( 0\leq s<1\right) ,\ a_{1},a_{2},...a_{n}\geq 0,\
b_{1},b_{2},...b_{n}\geq 0,\ $we obtain the inequality%
\begin{equation*}
\left\vert f\left( \frac{a+b}{2}\right) -\frac{1}{b-a}\dint%
\limits_{a}^{b}f(u)du\right\vert \leq \frac{b-a}{4}\left( \frac{4}{p+1}%
\right) ^{\frac{1}{p}}\left( \left\vert f^{\prime }(a)\right\vert
+\left\vert f^{\prime }(b)\right\vert \right)
\end{equation*}%
which is the same with the inequality (\ref{1-e}).
\end{remark}

\begin{theorem}
\label{2.4}Let $A\subset $ $%
\mathbb{R}
$ be an open invex subset with respect to $\eta :A\times A\rightarrow 
\mathbb{R}
$ and $a,b\in A$ with $a<a+\eta (b,a).$ Suppose that $f:A\rightarrow 
\mathbb{R}
$ is a differentiable function and $\left\vert f^{\prime }\right\vert ^{q}$
is preinvex function on $\left[ a,a+\eta (b,a)\right] $ for some fixed $%
q\geq 1$. If $f^{\prime }$ is integrable on $\left[ a,a+\eta (b,a)\right] $,
then the following inequality holds:%
\begin{eqnarray*}
&&\left\vert f(x)-\frac{1}{\eta (b,a)}\dint\limits_{a}^{a+\eta
(b,a)}f(u)du\right\vert \leq \eta (b,a)\left( \frac{1}{2}\right) ^{1-\frac{1%
}{q}} \\
&&\times \left\{ \left( \frac{x-a}{\eta (b,a)}\right) ^{2\left( 1-\frac{1}{q}%
\right) }\left[ \frac{\left( x-a\right) ^{2}\left( 3\eta (b,a)-2x+2a\right) 
}{6\eta ^{3}(b,a)}\left\vert f^{\prime }\left( a\right) \right\vert ^{q}+%
\frac{1}{3}\left( \frac{x-a}{\eta (b,a)}\right) ^{3}\left\vert f^{\prime
}\left( b\right) \right\vert ^{q}\right] ^{\frac{1}{q}}\right.  \\
&&\left. +\left( \frac{a+\eta (b,a)-x}{\eta (b,a)}\right) ^{2\left( 1-\frac{1%
}{q}\right) }\left[ \frac{1}{3}\left( \frac{a+\eta (b,a)-x}{\eta (b,a)}%
\right) ^{3}\left\vert f^{\prime }\left( a\right) \right\vert ^{q}+\left( 
\frac{1}{6}+\frac{\left( x-a\right) ^{2}\left( 2x-3\eta (b,a)-2a\right) }{%
6\eta ^{3}(b,a)}\right) \left\vert f^{\prime }\left( b\right) \right\vert
^{q}\right] ^{\frac{1}{q}}\right\} 
\end{eqnarray*}%
for each $x\in \left[ a,a+\eta (b,a)\right] .$
\end{theorem}

\begin{proof}
By Lemma \ref{2.1} and inequality (\ref{2-11}), and using the well known
power mean inequality, we have%
\begin{eqnarray*}
&&\left\vert f(x)-\frac{1}{\eta (b,a)}\dint\limits_{a}^{a+\eta
(b,a)}f(u)du\right\vert  \\
&\leq &\eta (b,a)\left( \dint\limits_{0}^{\frac{x-a}{\eta (b,a)}}tdt\right)
^{1-\frac{1}{q}}\left( \dint\limits_{0}^{\frac{x-a}{\eta (b,a)}}t\left\vert
f^{\prime }\left( a+t\eta (b,a)\right) \right\vert ^{q}dt\right) ^{\frac{1}{q%
}} \\
&&+\eta (b,a)\left( \dint\limits_{\frac{x-a}{\eta (b,a)}}^{1}\left(
1-t\right) dt\right) ^{1-\frac{1}{q}}\left( \dint\limits_{\frac{x-a}{\eta
(b,a)}}^{1}\left( 1-t\right) \left\vert f^{\prime }\left( a+t\eta
(b,a)\right) \right\vert ^{q}dt\right) ^{\frac{1}{q}} \\
&\leq &\eta (b,a)\left( \frac{1}{2}\right) ^{1-\frac{1}{q}}\left\{ \left( 
\frac{x-a}{\eta (b,a)}\right) ^{2\left( 1-\frac{1}{q}\right) }\left[ \frac{%
\left( x-a\right) ^{2}\left( 3\eta (b,a)-2x+2a\right) }{6\eta ^{3}(b,a)}%
\left\vert f^{\prime }\left( a\right) \right\vert ^{q}\right. \right.  \\
&&\left. +\frac{1}{3}\left( \frac{x-a}{\eta (b,a)}\right) ^{3}\left\vert
f^{\prime }\left( b\right) \right\vert ^{q}\right] ^{\frac{1}{q}}+\left( 
\frac{a+\eta (b,a)-x}{\eta (b,a)}\right) ^{2\left( 1-\frac{1}{q}\right) }%
\left[ \frac{1}{3}\left( \frac{a+\eta (b,a)-x}{\eta (b,a)}\right)
^{3}\left\vert f^{\prime }\left( a\right) \right\vert ^{q}\right.  \\
&&\left. \left. +\left( \frac{1}{6}+\frac{\left( x-a\right) ^{2}\left(
2x-3\eta (b,a)-2a\right) }{6\eta ^{3}(b,a)}\right) \left\vert f^{\prime
}\left( b\right) \right\vert ^{q}\right] ^{\frac{1}{q}}\right\} 
\end{eqnarray*}%
where we use the fact that%
\begin{equation*}
\dint\limits_{0}^{\frac{x-a}{\eta (b,a)}}tdt=\frac{1}{2}\left( \frac{x-a}{%
\eta (b,a)}\right) ^{2},
\end{equation*}%
\begin{eqnarray*}
&&\dint\limits_{0}^{\frac{x-a}{\eta (b,a)}}t\left\vert f^{\prime }\left(
a+t\eta (b,a)\right) \right\vert ^{q}dt \\
&\leq &\frac{\left( x-a\right) ^{2}\left( 3\eta (b,a)-2x+2a\right) }{6\eta
^{3}(b,a)}\left\vert f^{\prime }\left( a\right) \right\vert ^{q}+\frac{1}{3}%
\left( \frac{x-a}{\eta (b,a)}\right) ^{3}\left\vert f^{\prime }\left(
b\right) \right\vert ^{q},
\end{eqnarray*}%
\begin{equation*}
\dint\limits_{\frac{x-a}{\eta (b,a)}}^{1}\left( 1-t\right) dt=\frac{1}{2}%
\left( \frac{a+\eta (b,a)-x}{\eta (b,a)}\right) ^{2},
\end{equation*}%
\begin{eqnarray*}
&&\dint\limits_{\frac{x-a}{\eta (b,a)}}^{1}\left( 1-t\right) \left\vert
f^{\prime }\left( a+t\eta (b,a)\right) \right\vert ^{q}dt \\
&\leq &\frac{1}{3}\left( \frac{a+\eta (b,a)-x}{\eta (b,a)}\right)
^{3}\left\vert f^{\prime }\left( a\right) \right\vert ^{q}+\left( \frac{1}{6}%
+\frac{\left( x-a\right) ^{2}\left( 2x-3\eta (b,a)-2a\right) }{6\eta
^{3}(b,a)}\right) \left\vert f^{\prime }\left( b\right) \right\vert ^{q}.
\end{eqnarray*}%
The proof is completed.
\end{proof}

\begin{corollary}
\label{s2}Suppose that all the assumptions of Theorem \ref{2.4} are
satisfied. If we take $x=\frac{2a+\eta (b,a)}{2}$, then by the inequality (%
\ref{2-6}) we have the following inequality%
\begin{equation*}
\left\vert f\left( \frac{2a+\eta (b,a)}{2}\right) -\frac{1}{b-a}%
\dint\limits_{a}^{b}f(x)dx\right\vert \leq \left( \frac{3^{1-\frac{1}{q}}}{8}%
\right) \eta (b,a)\left( \left\vert f^{\prime }(a)\right\vert +\left\vert
f^{\prime }(b)\right\vert \right) .
\end{equation*}
\end{corollary}

\begin{remark}
Suppose that all the assumptions of Theorem \ref{2.4} are satisfied.

(a) In Corollary \ref{s2}, if we take $\eta (b,a)=b-a,$ then we have the
inequality 
\begin{equation*}
\left\vert f\left( \frac{a+b}{2}\right) -\frac{1}{b-a}\dint%
\limits_{a}^{b}f(x)dx\right\vert \leq \left( \frac{3^{1-\frac{1}{q}}}{8}%
\right) \left( b-a\right) \left( \left\vert f^{\prime }(a)\right\vert
+\left\vert f^{\prime }(b)\right\vert \right)
\end{equation*}%
which is the same with the inequality (\ref{1-ee}).

(b) If the mapping $\eta $ satisfies condition C then by use of the
inequality (\ref{2-11}) for $\left\vert f^{\prime }\right\vert ^{q}$ in the
proof of Theorem \ref{2.4},we get%
\begin{eqnarray*}
&&\left\vert f(x)-\frac{1}{\eta (b,a)}\dint\limits_{a}^{a+\eta
(b,a)}f(u)du\right\vert \leq \eta (b,a)\left( \frac{1}{2}\right) ^{1-\frac{1%
}{q}} \\
&&\times \left\{ \left( \frac{x-a}{\eta (b,a)}\right) ^{2\left( 1-\frac{1}{q}%
\right) }\left[ \frac{\left( x-a\right) ^{2}\left( 3\eta (b,a)-2x+2a\right) 
}{6\eta ^{3}(b,a)}\left\vert f^{\prime }\left( a\right) \right\vert
^{q}\right. \right.  \\
&&\left. +\frac{1}{3}\left( \frac{x-a}{\eta (b,a)}\right) ^{3}\left\vert
f^{\prime }\left( a+\eta (b,a)\right) \right\vert ^{q}\right] ^{\frac{1}{q}%
}+\left( \frac{a+\eta (b,a)-x}{\eta (b,a)}\right) ^{2\left( 1-\frac{1}{q}%
\right) }\left[ \frac{1}{3}\left( \frac{a+\eta (b,a)-x}{\eta (b,a)}\right)
^{3}\left\vert f^{\prime }\left( a\right) \right\vert ^{q}\right.  \\
&&\left. \left. +\left( \frac{1}{6}+\frac{\left( x-a\right) ^{2}\left(
2x-3\eta (b,a)-2a\right) }{6\eta ^{3}(b,a)}\right) \left\vert f^{\prime
}\left( a+\eta (b,a)\right) \right\vert ^{q}\right] ^{\frac{1}{q}}\right\} 
\end{eqnarray*}
\end{remark}

\end{document}